\documentclass{amsart}

\usepackage{ifthen}
\usepackage{graphicx}
\usepackage{amssymb}
\usepackage{enumerate}
\usepackage{url}
\newtheorem{anyprop}{Anyprop}[section]

\newtheorem{theorem}[anyprop]{Theorem}

\newtheorem{corollary}[anyprop]{Corollary}

\theoremstyle{definition}

\newtheorem{example}[anyprop]{Example}

\newcommand  {\Spec}    {\operatorname{Spec}}

\theoremstyle{remark}

\numberwithin{equation}{section}

\usepackage{amscd}
\usepackage{amssymb}
\input xy
\xyoption{all}

\begin{document}
\title[ON SPECIAL PROPERTIES INVOLVING THE NOETHERIAN CONDITION]
{ON PRESERVATION PROPERTIES AND A SPECIAL ALGEBRAIC CHARACTERIZATION OF SOME STRONGER FORMS OF THE NOETHERIAN CONDITION}


\author[Danny A. J. G\'omez-Ram\'irez]{Danny Arlen de Jes\'us G\'omez-Ram\'irez}
\author[Juan D. Velez]{Juan D. V\'elez}
\author[Edisson Gallego]{Edisson Gallego}
\address{Vienna University of Technology, Institute of Discrete Mathematics and Geometry,
wiedner Hauptstaße 8-10, 1040, Vienna, Austria.}
\address{Universidad Nacional de Colombia, Escuela de Matem\'aticas,
Calle 59A No 63 - 20, N\'ucleo El Volador, Medell\'in, Colombia.}
\address{University of Antioquia, Calle 67 \# 53-108, Medell\'in, Colombia.}
\email{daj.gomezramirez@gmail.com}
\email{jdvelez@unal.edu.co}
\email{egalleg@gmail.com}



\begin{abstract}
We give an elementary proof prove of the preservation of the Noetherian condition for commutative rings with unity $R$ having at least one finitely generated ideal $I$ such that the quotient ring is again finitely generated, and $R$ is $I-$adically complete. Moreover, we offer as a direct corollary a new elementary proof of the fact that if a ring is Noetherian then the corresponding ring of formal power series in finitely many variables is Noetherian. In addition, we give a counterexample showing that the `completion' condition cannot be avoided on the former theorem. Lastly, we give an elementary characterization of Noetherian commutative rings that can be decomposed as a finite direct product of fields.

\end{abstract}

\maketitle


\smallskip

\newline\indent

\section*{Introduction}
Among the most studied rings in commutative algebra are the Noetherian ones, i.e., commutative rings with unity such that every ideal can be finitely generated. Moreover, one needs to increase a little bit the level of formal sophistication in order to find simple examples of non-Noetherian structures such as the ring of polynomials in countable many variables over a field or the ring of algebraic integers. So, (elementary) results preserving and implying the Noetherian condition after the application of standard algebraic operations (such as completion, quotient, localization, etc.) are quite useful (see \cite{kunz}, \cite{eisenbudSubrings}, \cite{richman} and \cite{jategaonkar}).
Similarly, new characterizations of (stronger forms of the) Noetherian condition are oft very valuable for enlightening our understanding of what finite generation in (non-)commutative algebra means (see \cite{van}, \cite{vamos}, \cite{bass} and \cite{he}).

Finally, we will prove in an elementary way two results concerning, on the one hand, the preservation of the Noetherian condition for a commutative ring with unity $R$ having at least one finitely generated ideal $I$ such that the quotient ring is again finitely generated and $R$ is $I-$adically complete. In addition, we also offer as a corollary a new elementary proof of the fact that the ring of formal power series in finitely many variables is Noetherian, if its ring of coefficient so is, and we give a counterexample showing that the `completion' condition cannot be avoided. On the other hand, we give a quite simple algebraic characterization of Noetherian commutative rings that can be decomposed as a finite direct product of fields.

\section{Preservation Properties of the Noetherian Condition involving a Special Class of Finitely Generated Ideals}

In this section, by a ring we will mean a commutative ring with identity,
not necessarily Noetherian.

Let $R$ be a ring, and let $I$ be an ideal of $R$. We will show that
if $I$ is finitely generated and $R/I$ is Noetherian, then the completion
of $R$ with respect to the $I$-adic topology is also Noetherian. In general
one cannot expect $R$ to be Noetherian under these hypothesis, as shown in
Example \ref{contraejemplo} below. This result can be regarded as a
generalization of the following well known corollary to the Cohen Structure
Theorems (\cite{cohen}), (\cite[pag. 189,201]{eisenbud}): if $(R,m)$ is a
quasilocal complete ring (with respect to the $m$-adical topology) then $R$
is Noetherian if $m$ is finitely generated. The Noetherian property is
deduced from the fact that under these hypothesis $R$ is a quotient of a
power series ring over a complete discrete valuation ring. We observe that
this result can be recovered immediately from Theorem \ref{principal} as the
very special case when $I$ is maximal, since $R/I$ is a field, hence
automatically Noetherian. While Cohen Structure Theorems require some
machinery, the result below is totally elementary.

As a corollary, we deduce a quite elementary new proof of the fact that If $%
A $ is Noetherian, so it is the power series ring $A[[x_{1},\ldots ,x_{n}]].$

\begin{theorem}
\label{principal}Let $R$ be a ring, and let $I$ be a finitely generated
ideal of $R$. Suppose that $R$ is $I$-adically complete, and that $R/I$ is
Noetherian. Then $R$ is also Noetherian.
\end{theorem}

\begin{proof}
If $R$ were not Noetherian, a standard argument using Zorn's Lemma shows
that there is a maximal ideal $P$ in $R$ with respect to the property of not
being finitely generated, and this ideal is necessarily prime. It is clear
that $P$ does not contain $I,$ otherwise, since $R/I$ is Noetherian, any
lifting of a set of generator for $P/I$ coupled with generators of $I$ would
also generate $P.$ Thus, we may assume that there exist $z\in I$ such that $%
z\notin P.$ By the maximality of $P$, the ideal $P+Rz$ must be finitely
generated. Let $f_{i}+r_{i}z,$ $i=1,\ldots ,n$ be any set of generators,
with $f_{i}\in P.$

We claim that $\{f_{1},\ldots ,f_{n}\}$ is a set of generators for $P.$ Let $%
f$ be any element of $P$. Then, since $f$ is a priory in $P+Rz$ there must
be elements $g_{i}^{(0)}\in R$ such that 
\begin{equation}
f=\sum_{i=1}^{n}g_{i}^{(0)}(f_{i}+r_{i}z)=\sum_{i=1}^{n}g_{i}^{(0)}f_{i}+zf^{(1)},  \label{e1}
\end{equation}%
where $f^{(1)}=\sum_{i=1}^{n}g_{i}^{(0)}r_{i}.$ Hence, $%
f-\sum_{i=1}^{n}g_{i}^{(0)}f_{i}=zf^{(1)}$ is in $P.$ Since we are
assuming $z\notin P,$ and $P$ is prime we deduce $f^{(1)}\in P.$

The same reasoning applied to $f^{(1)}$ yields $f^{(1)}=\sum{i=1}^{n}g_{i}^{(1)}f_{i}+zf^{(2)}$, for some elements $g_{i}^{(1)}\in
R,$ and $f^{(2)}\in P.$ Replacing $f^{(1)}$ in \ref{e1} by the right hand
side of this last equation gives $f=\sum_{i=1}^{n}(g_{i}^{(0)}+zg_{i}^{(1)})f_{i}+z^{2}f^{(2)}.$ Since $%
z^{t}\notin P$ for any $t>0,$ a straightforward induction shows that $f$ can
be written as 
\begin{equation}
f=\sum_{i=1}^{n}(g_{i}^{(0)}+zg_{i}^{(1)}+\cdots
+z^{t}g_{i}^{(t)})f_{i}+z^{t+1}f^{(t+1)},\text{ }  \label{e2}
\end{equation}

for certain elements $g_{i}^{(t)}\in R,$ and $f^{(t+1)}\in P.$ Since $R$ is
complete, $h_{i}=\sum_{t=0}^{\infty }z^{t}g_{i}^{(t)}$ is a well
defined element of $R.$

We now observe that $R$ must be $I$-adically separated, i.e., $%
\bigcap_{t=0}^{\infty }I^{t}$ is just the kernel of the canonical
map $i:R\longrightarrow R^{\widehat{I}}$, which should be the zero ideal, because $i$ is an isomorphism, since we are making
the assumption that $R$ is $I$-adically complete. From this, we deduce that 
$f=\sum_{i=1}^{n}h_{i}f_{i},$ since 

\[f-\sum_{i=1}^{n}h_{i}f_{i}=f-\sum_{i=1}^{n}(g_{i}^{(0)}+zg_{i}^{(1)}+\cdots
+z^{t}g_{i}^{(t)})f_{i}-z^{t+1}\sum_{i=1}^{n}\left(
\sum_{j=0}^{\infty }z^{j}g_{i}^{(j+t+1)}\right) f_{i}\]

\[=z^{t+1}\left(f^{(t+1)}-\sum_{i=1}^{n}\left( \sum_{j=0}^{\infty }z^{j}g_{i}^{(j+t+1)}\right) f_{i} \right),\]

is an element of $I^{t},$ for all $t$.

In conclusion, any element $f\in P$ can be generated by the set $\{f_1,\cdots,f_n\}\subseteq P$. Therefore, $P$ would be finitely generated, which contradicts our former assumption.
\end{proof}

\begin{corollary}
\label{nuevo}If $A$ is Noetherian, so it is $A[[x_{1},\ldots ,x_{n}]].$
\end{corollary}

\begin{proof}
Let $R=A[[x_{1},\ldots ,x_{n}]]$ and $I=(x_{1},\ldots ,x_{n})$. A standard
argument (\cite[pag. 192]{eisenbud}), (\cite[pag. 61]{matcomrintheo}) shows that 
$R=B^{\widehat{I}}$ where $B=A[x_{1},\ldots ,x_{n}]$. Since $R/IR\simeq A$
is Noetherian, the previous theorem gives that so it is $R.$
\end{proof}

In general, if $R$ is not complete with respect to the $I$-adic topology, it
is not true that $R$ is Noetherian under the hypothesis of $I$ being
finitely generated and $R/I$ being Noetherian, not even in the case where $I$
is maximal, as the following example shows.

\begin{example}
\label{contraejemplo}

\emph{Let }$\mathbb{N}$ \emph{denote the set of natural
numbers, and let }$U$\emph{\ be any non principal ultrafilter in }$N$\emph{,
that is, a collection of infinite subsets of }$N$\emph{, closed under finite
intersection, with the property that for any }$D\subset N$\emph{, either }$D$%
\emph{\ or its complement belongs to }$U.$\emph{\ Let }$(R,m)$\emph{\ be any
discrete valuation ring, and let us denote by }$R_{w}$\emph{\ a copy of }$R$%
\emph{\ indexed by the natural }$w\in N.$\emph{\ By }$S$\emph{\ we will
denote the ultraproduct, }$S=u\lim_{w\rightarrow \infty }R_{w}.$\emph{\ We
recall that this is defined as the set of equivalent classes in the
Cartesian product }$\prod_{w\in \mathbb{N}}R_{w}$\emph{, where two
sequences }$(a_{w})$\emph{\ and }$(b_{w})$\emph{\ are regarded as equivalent
if the set of indices }$w$\emph{\ where }$a_{w}=b_{w}$\emph{\ is an element
of }$U.$\emph{\ This is a ring with the obvious operations, and it is also
local with a principal maximal ideal }$m'$\emph{\ generated by the class of }$%
(p_{w})$\emph{, where }$m_{w}=(p_{w})$\emph{\ is the maximal ideal of }$%
R_{w} $\emph{\ (see \cite[Ch. 1-2]{schoutens}). If }$c$\emph{\ denotes
the class of the sequence of powers }$(p_{w})^{w}$\emph{, then it is clear
that }$c$\emph{\ belongs to the Jacobson radical of }$S,$\emph{\ }$\cap
_{w=0}^{\infty }(m')^{w},$\emph{\ and it is a nonzero element. Consequently, }$%
S $\emph{\ cannot be Noetherian, even though its maximal ideal is finitely
generated (actually, principal), and }$S/m'$\emph{\ is a field. }
\end{example}

\section{A Characterization of a Noetherian Finite Direct Product of Fields}

In this section we will give an elementary algebraic re-formulation of the fact that a commutative Noetherian ring is the direct product of fields by means of a idempotent-membership condition, namely, the fact that any of its elements belongs lo the ideal generated by its-own square.
\begin{theorem}
Let $R$ be a commutative Noetherian ring. Then $R$ is the finite direct product of fields if and only if any element $f\in R$, holds that $f\in (f^2)$.
\end{theorem}
\begin{proof}
If $R$ is a finite product of fields, then clearly the desired condition is satisfied, since any element in $R$ is the direct product of zeros and unities.

Conversely, let us assume, by contradiction, that $R$ is a Noetherian ring which is not a finite product of fields. We want to prove 
that there is an element $f\in R$ such that $f\notin (f^2)$. 
 In fact, we can reduced to the case of ${\Spec}R$ connected, because if ${\rm Spec}R$ is not connected then, due to the Noetherian hypothesis, we
 can write ${\rm Spec}R=V(Q_1)\uplus\cdots\uplus V(Q_s)$, where $V(Q_j)\cong {\rm Spec}(R/Q_j)$ are the connected components of ${\rm Spec}R$.
 Hence, by the Chinese Remainder Theorem \cite{atimac}, $R\cong \prod_{i=1}^sR/Q_i$ and by the previous assumption at least one of the $R/Q_i$ is not a field. So, it is enough to find an $f_i\in R/Q_i$ such that $f_i\notin (f_i^2)$ to obtain the desired element $f=(0,...,f_i,...,0)\in R$.
Now, the connectedness of ${\rm Spec}R$ it is equivalent to saying that the only idempotents
 of $R$ are trivial ones, namely, zero and one (see for example \cite[Ch. 2]{hartshorne}).

Lastly, choose $f\in R$ neither a unit nor idempotent. Then, $f\notin (f^2)$. In fact, by contradiction, if $f=cf^2$, for some $c\in R$, and
 so $cf(1-cf)=0$, which means that $cf$ is idempotent. Hence, $cf=0$ or $cf=1$. In the first case we have $f=(cf)f=0$,
 and in the second case, $f$ is a unit. Then both cases contradicts our hypothesis on $f$.
 \end{proof}

\section*{Acknowledgements}
The authors want to thank the Universidad Nacional of Colombia for its support. In addition, Danny A. J. G\'omez-Ram\'irez specially thanks Holger Brenner for all the inspiring discussions during the preparation of this work. D. A. J. G\'omez-Ram\'irez was supported by the Vienna Science and Technology Fund (WWTF) as part of the Vienna Research Group 12-004. Finally, he thanks deeply to Rubents Ram\'irez for being \emph{nuestro milagro de vida} and for all the inspiration, to Veronica Ramirez for her unconditional support and to J. Quintero for (the determination of) being there.



\end{document}